\documentclass[preprint,12pt]{elsarticle}

 \usepackage{graphicx}

\usepackage{amssymb}
\usepackage{amsthm}
\usepackage{amsmath}
\usepackage{amsfonts}

\usepackage{setspace}
\usepackage{natbib}
\biboptions{sort&compress, square}

\newtheorem{theorem}{Theorem}

\newtheorem{corollary}[theorem]{Corollary}
\newtheorem{definition}[theorem]{Definition}

\newtheorem{example}[theorem]{Example}

\newtheorem{proposition}[theorem]{Proposition}
\newtheorem{remark}[theorem]{Remark}
\date{}
\journal{ }

\begin{document}

\begin{frontmatter}


\author{Mohammed-Salah Abdelouahab\corref{*}}
\cortext[*]{corresponding author}
\ead{medsala3@yahoo.fr}
\author{Nasr-Eddine Hamri}
\ead{n.hamri@centre-univ-mila.dz}


\title{The Gr$\ddot{\text{u}}$nwald-Letnikov fractional-order derivative with fixed memory length}


\address{Department of Mathematics and Computer sciences\\
Mila University Centre, 43000, Algeria}

\begin{abstract}
Contrary to integer order derivative, the fractional-order derivative of a non-constant periodic function is not a periodic function with the same period, as a consequence of this property the time-invariant fractional order system does not have any non-constant periodic solution unless the lower terminal of the derivative is $\pm\infty$, which is not practical. This property limits the applicability areas of fractional derivatives and makes it unfavorable, for a wide range of periodic real phenomena. Therefore enlarging the applicability of fractional system to such real area is an important research topic.
In this paper we attempt to give a solution for the above problem by imposing a simple modification on the Gr$\ddot{\text{u}}$nwald-Letnikov definition of fractional derivative, this modification consists of fixing the memory length and varying the lower terminal of the derivative. It is shown that the new proposed definition of fractional derivative preserves the periodicity.

\end{abstract}


\end{frontmatter}

\section{Introduction}
In 1695, Leibniz laid the foundations of fractional-order derivative which means the extension of integer-order derivative concept \cite{0}, but the first systematic studies seem to have been made at the beginning and middle of the nineteenth century by Liouville, Riemann, and Holmgren. Liouville has expanded functions in series of exponentials and defined the $n$th-order derivative of such a series by operating term-by-term as though $n$ were a positive integer. Riemann proposed a different definition that involved a definite integral and was applicable to power series with non-integer exponents. It was Gr$\ddot{\text{u}}$nwald and
Krug who first unified the results of Liouville and Riemann. Gr$\ddot{\text{u}}$nwald,
by returning to the original sources and adopting as starting point the
definition of a derivative as the limit of a difference quotient and arriving at
definite-integral formulas for the $n$th-order derivative. Krug, working through
Cauchy's integral formula for ordinary derivatives, showed that Riemann's
definite integral had to be interpreted as having a finite lower limit while
Liouville's definition corresponded to a lower limit $-\infty$ \cite{2, 2*,3,4,0*}. 
It turns out that the Riemann-Liouville derivatives have certain disadvantages when
trying to model real-world phenomena with fractional differential equations. In the second half of the twentieth century Caputo proposed a modified concept of a fractional derivative \cite{1}, by inverting the composite order of integer-order derivative operator and integral operator arising in the  Riemann-Liouville definition, the resulting definition is better suited to such tasks.
It has been found that many systems in interdisciplinary fields can be described
by the fractional differential equations, such as viscoelastic
systems, dielectric polarization, electrode-electrolyte polarization, electromagnetic waves, and quantum
evolution of complex systems \cite{a,b,c,d,e}.
Due to the growing interest of fractional-order derivatives to be applied in different areas, it seems analysis of this type of derivative is of great importance.
The existence of periodic solutions is often a desired property in dynamical systems, constituting one of the most important research directions in the theory of dynamical systems, with applications ranging from celestial mechanics to biology and finance. 
In chaos theory, the idea of chaos control is based on the fact that
chaotic attractors have a skeleton made of an infinite number
of unstable periodic orbits which are subjected for stabilization
\cite{e1}. Thus for chaos control in fractional
order systems, it is important to show that periodic solutions exist
in such systems. Tavazoei \cite{f} have proved that the fractional-order derivatives of a periodic function with a specific period cannot be a periodic function with the same period, as a consequence, the periodic solution cannot be detected in fractional-order systems, under any
circumstances \cite{g,h,i,j,k,l}. This property limits the applicability areas of fractional-order systems and makes it unfavorable, for a wide range of periodic real phenomena. Therefore enlarging the applicability of fractional-order systems to such real area is an important research topic. The most important contribution of this paper is to this end.
\section{Gr$\ddot{\text{u}}$nwald-Letnikov fractional derivative}

Lets $\alpha>0$, the Gr$\ddot{\text{u}}$nwald-Letnikov $\alpha$th order fractional derivative of function $f(t)$  with respect to $t$ and the terminal value $a$ is given by \cite{2}

\begin{equation} \label{e22}
\ ^{GL}_{a}D_{t}^{\alpha}f(x)=\lim\limits_{%
\begin{array}
[c]{c}%
h\rightarrow0\\
nh=x-a
\end{array}
}%
h^{-\alpha}\sum\limits_{k=0}^{n}(-1)^{k}\left( 
\begin{array}
[c]{c}%
\alpha\\
k
\end{array}
\right) f(x-kh),
\end{equation}

where
$$\left( 
\begin{array}
[c]{c}%
\alpha\\
k
\end{array}
\right)=\dfrac{\alpha(\alpha-1)(\alpha-2)...(\alpha-k+1)}{k!}=\dfrac{\Gamma(\alpha+1)}{k!\Gamma(\alpha-k+1)},
$$ 

\begin{equation} \label{e23}
\ ^{GL}_{a}D_{t}^{\alpha}f(x)=\lim\limits_{%
h\rightarrow0\\
}%
h^{-\alpha}\sum\limits_{k=0}^{\frac{x-a}{h}}(-1)^{k}\left( 
\dfrac{\Gamma(\alpha+1)}{k!\Gamma(\alpha-k+1)}
\right) f(x-kh),
\end{equation}

where $\Gamma$ is the gamma function defined by the Euler limit expression

\begin{equation}\label{e0}
\Gamma(x)=\lim\limits_{n\rightarrow \infty}\dfrac{n!n^{x}}{x(x+1)...(x+n)}
\end{equation}
where $x>0$.

or the so-called Euler integral definition:
\begin{equation}
\Gamma(x)=\int^{\infty}_{0}	t^{x-1}e^{-t}dt \  \  \  (x>0),
\end{equation}

The following theorem gives negative answer for question of preservation of the periodicity \cite{f}

\begin{theorem}
Suppose that $f(t)$ is $(m - 1)$-times continuously
differentiable and $f^{(m)}$ is bounded. If $f(t)$ is a non-constant
periodic function with period $T$,then the functions $\ ^{GL}_{\ a}D_{t}^{\alpha}f(t)$, 
where $0<\alpha \notin \textbf{N}$ and $m$ is the first integer greater than $\alpha$, cannot be
periodic functions with period $T$.
\end{theorem}
\begin{example}
The derivative of the sine function is given by 
\begin{align*}
\ ^{GL}_{\ 0}D_{t}^{\alpha}sin(t)=t^{1-\alpha}E_{2,2-\alpha}(-t^{2}),
\end{align*}
where $E_{\alpha,\beta}$ is the two-parameter Mittag-Laffler function defined by \cite{2}
\begin{align*}
E_{\alpha,\beta}(t)=\sum\limits_{k=0}^{\infty} 
\dfrac{t^{k}}{\Gamma(\alpha k+\beta)}.
\end{align*}
\end{example}
Figure \ref{fg1}  illustrate numerical approximation of $\ ^{GL}_{\ 0}D_{t}^{0.5}sin(t)$ which is not periodic but it converges to the periodic function $sin(t+\alpha \frac{\pi}{2})$.

As a consequence of the above theorem, the periodic solution cannot be detected in fractional-order systems, under any circumstances \cite{f,g}.
\begin{corollary}
A differential equation of fractional-order in the form
\begin{align*}
\ ^{GL}_{\ a}D_{t}^{\alpha}x(t)=f(x(t)),
\end{align*}
where $0<\alpha\notin\textbf{N}$, cannot have any non-constant smooth periodic solution.
\end{corollary}
This property makes fractional systems not applicable for a wide range of real periodic phenomena.
\section{Gr$\ddot{\text{u}}$nwald-Letnikov fractional derivative with fixed memory length}
In this section we introduce a new definition of fractional derivative based on the Gr$\ddot{u}$nwald-Litnikov definition and we call it the Gr$\ddot{\text{u}}$nwald-Letnikov fractional derivative with fixed memory length.

\begin{definition}\label{df1} 
Let $\alpha\geq 0$, $L\geq 0$,  $m$ an integer such that $m-1\leq \alpha<m$ and $f$ an integrable function in the interval $[a-L,b]$. We define the Gr$\ddot{\text{u}}$nwald-Letnikov fractional derivative with fixed memory length $L$ of $f$,  $ \ ^{^{MG}}_{_{\ \ L}}D_{t}^{\alpha}f$ by  
\begin{equation} \label{e1}
_{_{\ \ L}}^{^{MG}}D_{t}^{\alpha}f(t)=\lim_{h\rightarrow 0}\dfrac{1}{h^{\alpha}}\sum\limits_{k=0}^{\frac{L}{h}}(-1)^{k}\dfrac{\Gamma(\alpha+1)}{k!\Gamma(\alpha-k+1)}f(t-kh), \ \ t\in [a,b].
\end{equation}
\end{definition}

\begin{remark}
A remarkable difference between classical definitions of fractional-order derivative and the new definition is that in the  classical definitions the lower terminal $a$ is fixed and the memory length $t-a$ is a function of the variable $t$  but in the new definitions the lower terminal $a=t-L$ is a function of the variable $t$  and the memory length $L$ is fixed.
\end{remark}
The definition of the Gr$\ddot{\text{u}}$nwald-Letnikov fractional derivative with fixed memory length $L$ of $f$, given in definition \textbf{(\ref{df1})} is expressed in term of a limit which is not suitable for analytic study. The following proposition gives an evaluation of this limit under some restrictions of the function $f$.
\begin{proposition}
Under the assumptions of definition \textbf{(\ref{df1})} if the function $f$ is $(m+1)$-differentiable with $f^{(m+1)}\in L_{1}[a-L,b]$ we have 
\begin{equation} \label{e6}
\ _{\ \ L}^{^{MG}}D_{t}^{\alpha}f(t)=\sum\limits_{k=0}^{m}\dfrac{f^{(k)}(t-L)L^{k-\alpha}}{\Gamma(k-\alpha+1)}+\dfrac{1}{\Gamma(m-\alpha+1)}\int\limits_{t-L}^{t}(t-\tau)^{m-\alpha}f^{(m+1)}(\tau)d\tau.  
\end{equation}
and
\begin{equation} \label{e6}
\ _{\ \ L}^{^{MG}}D_{t}^{\alpha}f(t)=\dfrac{1}{\Gamma(2m-\alpha+1)}\int\limits_{t-L}^{t}(t-\tau)^{2m-\alpha}f^{(m+1)}(\tau)d\tau.  
\end{equation}
\end{proposition}

\begin{proof}
Our aim is to evaluate the limit 
\begin{equation} \label{e11}
_{_{\ \ L}}^{^{MG}}D_{t}^{\alpha}f(t)=\lim_{h\rightarrow 0}\dfrac{1}{h^{\alpha}}\sum\limits_{k=0}^{\frac{L}{h}}(-1)^{k}\dfrac{\Gamma(\alpha+1)}{k!\Gamma(\alpha-k+1)}f(t-kh)=\lim_{h\rightarrow 0 }   f^{(\alpha)}_{h}(t),
\end{equation}
where 
\begin{equation}
f^{(\alpha)}_{h}(t)=h^{-\alpha}\sum\limits_{k=0}^{n}(-1)^{k}\dfrac{\Gamma(\alpha+1)}{k!\Gamma(\alpha-k+1)}f(t-kh), \ n=\dfrac{L}{h}.
\end{equation}
Noting 
\begin{equation}
\left( \begin{array}{c}
\alpha       \\
k
\end{array}\right)
=\dfrac{\Gamma(\alpha+1)}{k!\Gamma(\alpha-k+1)}
\end{equation}
then we have 
\begin{equation} \label{e12}
\left( \begin{array}{c}
\alpha       \\
k
\end{array}\right)
=
\left( \begin{array}{c}
\alpha -1      \\
k
\end{array}\right)
+\left( \begin{array}{c}
\alpha -1      \\
k-1
\end{array}\right)
\end{equation}

it follows that
 
\begin{eqnarray}
f^{(\alpha)}_{h}(t)=h^{-\alpha}\sum\limits_{k=0}^{n}(-1)^{k}\left( \begin{array}{c}
\alpha -1    \\
k
\end{array}\right)f(t-kh)+
h^{-\alpha}\sum\limits_{k=1}^{n}(-1)^{k}\left( \begin{array}{c}
\alpha -1      \\
k-1
\end{array}\right)f(t-kh) \ \ \ \ \ \ \ \ \ \ \ \\
=h^{-\alpha}\sum\limits_{k=0}^{n}(-1)^{k}\left( \begin{array}{c}
\alpha -1      \\
k
\end{array}\right)f(t-kh)+h^{-\alpha}\sum\limits_{k=0}^{n-1}(-1)^{k+1}\left( \begin{array}{c}
\alpha -1      \\
k
\end{array}\right)f(t-(k+1)h)\\
=(-1)^{n}\left( \begin{array}{c}
\alpha -1      \\
n
\end{array}\right)h^{-\alpha}f(t-L)
+h^{-\alpha}\sum\limits_{k=0}^{n-1}(-1)^{k}\left( \begin{array}{c}
\alpha -1      \\
k
\end{array}\right)\Delta f(t-kh), \label{e13}\ \ \ \ \ \ \ \ \ \ \ \ \ 
\end{eqnarray}

where we denote 
\begin{equation*}
\Delta f(t-kh)=f(t-kh)-f(t-(k+1)h).
\end{equation*}
Applying the property (\ref{e12}) of binomial coefficients repeated $m$ times, starting from \textbf{(\ref{e13})} we obtain:

\begin{eqnarray*}
f^{(\alpha)}_{h}(t)=(-1)^{n}\left( \begin{array}{c}
\alpha -1      \\
n
\end{array}\right)h^{-\alpha}f(t-L)
+(-1)^{n-1}\left( \begin{array}{c}
\alpha -2      \\
n-1
\end{array}\right)h^{-\alpha}\Delta f(t-L+h)\\ +h^{-\alpha}\sum\limits_{k=0}^{n-2}(-1)^{k}\left( \begin{array}{c}
\alpha -2      \\
k
\end{array}\right)\Delta^{2} f(t-kh), \label{e15}\ \ \ \ \ \ \ \ \ \ \ \ \ \ \ \ \ \ \ \ \ \ \ \ \ \ \ \ \ \ \ \ \  \\
=...\ \ \ \ \ \ \ \ \ \ \ \ \ \ \ \ \ \ \ \ \ \ \ \ \ \ \ \ \ \ \ \ \  \ \ \ \ \ \ \ \ \ \ \ \ \ \ \ \ \ \ \ \ \ \ \ \ \ \ \ \ \ \ \ \ \ \ \ \ \ \ \ \ \ \ \ \ \ \ \ \ \ \ \ \ \ \ \ \ \ \ \ \ 
\end{eqnarray*}

\begin{align}
=h^{-\alpha}\sum\limits_{k=0}^{m}(-1)^{n-k}\left( \begin{array}{c}
\alpha -k-1 \\
n-k
\end{array}\right)\Delta^{k} f(t-L+kh) \label{e14} \ \ \ \ \ \ \ \ \ \ \ \ \ \ \\  +h^{-\alpha}\sum\limits_{k=0}^{n-m-1}(-1)^{k}\left( \begin{array}{c}
\alpha -m-1      \\
k
\end{array}\right)\Delta^{m+1} f(t-kh) \label{e16}  \ \ \ \ \ \ \ \ \ \ \  
\end{align}
Let us evaluate the limit of the $k$th term in the sum (\ref{e14})

\begin{align*}
\lim_{h\rightarrow 0} h^{-\alpha}(-1)^{n-k}\left( \begin{array}{c}
\alpha -k-1 \\
n-k
\end{array}\right)\Delta^{k} f(t-L+kh) \ \ \ \ \ \ \ \ \ \ \ \ \ \ \ \ \ \ \ \ \ \ \ \ \ \ \ \ \\
= \lim_{h\rightarrow 0} (-1)^{n-k}\left( \begin{array}{c}
\alpha -k-1 \\
n-k
\end{array}\right)(n-k)^{\alpha-k} \ \ \ \ \ \ \ \ \ \ \ \ \ \ \ \ \ \ \ \ \ \ \ \ \ \ \ \  \\
\times(\dfrac{n}{n-k})^{\alpha-k}(nh)^{k-\alpha}\dfrac{\Delta^{k} f(t-L+kh)}{h^{k}}, \ \ n=\dfrac{L}{h}\rightarrow \infty \\
=L^{k-\alpha}\lim_{n\rightarrow \infty} (-1)^{n-k}\left( \begin{array}{c}
\alpha -k-1 \\
n-k
\end{array}\right)(n-k)^{\alpha-k} \ \ \ \ \ \ \ \ \ \ \ \ \ \ \ \ \ \ \ \  \\
\times \lim_{n\rightarrow \infty}(\dfrac{n}{n-k})^{\alpha-k}\lim_{h\rightarrow 0} \dfrac{\Delta^{k} f(t-L+kh)}{h^{k}},\ \ \ \ \ \ \ \ \ \ \ \ \ \ 
\end{align*}
\begin{align}
=\dfrac{f^{(k)}(t-L)L^{k-\alpha}}{\Gamma(k-\alpha+1)}, \label{e17}\ \ \ \ \ \ \ \ \ \ \ \ \ \ \ \ \ \ \ \ \ \ \ \ \ \ \ \ \ \ \ \ \ \ \ \ \ \ \ \ \ \ \ \ \ \ \ \ \
\end{align}
because we have
$$
\lim_{n\rightarrow \infty} (\dfrac{n}{n-k})^{\alpha-k}=1,\ \ \ \
\lim_{h\rightarrow 0} \dfrac{\Delta^{k} f(t-L+kh)}{h^{k}}=f^{(k)}(t-L),
$$
and from (\ref{e0}) we have
$$
\lim_{n\rightarrow \infty} (-1)^{n-k}\left( \begin{array}{c}
\alpha -k-1 \\
n-k
\end{array}\right)(n-k)^{\alpha-k}=\dfrac{1}{\Gamma(k-\alpha+1)}.
$$
To evaluate the limit of  the sum (\ref{e16}) let us write it in the form
$$
\dfrac{1}{\Gamma(m-\alpha+1)}\sum\limits_{k=0}^{n-m-1}(-1)^{k}\Gamma(m-\alpha+1)\left( \begin{array}{c}
\alpha -m-1      \\
k
\end{array}\right)k^{\alpha-m}
h(kh)^{m-\alpha}\dfrac{\Delta^{m+1} f(t-kh)}{h^{m+1}}. $$
Taking the notation  

\begin{align*}
\beta_{k}=(-1)^{k}\Gamma(m-\alpha+1)\left( \begin{array}{c}
\alpha -m-1      \\
k
\end{array}\right)k^{\alpha-m},\\
\alpha_{n,k}=h(kh)^{m-\alpha}\dfrac{\Delta^{m+1} f(t-kh)}{h^{m+1}}, \ \ \ h=\dfrac{L}{n}.  \ \ \ \ \ \ \ 
\end{align*}
Using (\ref{e0}) we obtain
\begin{equation} \label{e18}
\lim_{k\rightarrow\infty} \beta_{k}=1.
\end{equation}
Furthermore, if $m-\alpha>-1$, then 

\begin{align*}
\lim_{n\rightarrow\infty}\sum\limits_{k=0}^{n-m-1}\alpha_{n,k}=\lim_{h\rightarrow 0}\sum\limits_{k=0}^{n-m-1}h(kh)^{m-\alpha}\dfrac{\Delta^{m+1} f(t-kh)}{h^{m+1}},
\end{align*}
\begin{align}\label{e19}
\ \ \ \ \  \ \ \ \ \ =\int\limits_{t-L}^{t}(t-\tau)^{m-\alpha}f^{(m+1)}(\tau)d\tau.
\end{align}
Taking into account (\ref{e18}) and (\ref{e19}) and applying theorem (2.1) in \cite{2} we obtain 

\begin{align*}
\lim_{h\rightarrow 0}h^{-\alpha}\sum\limits_{k=0}^{n-m-1}(-1)^{k}\left( \begin{array}{c}
\alpha -m-1      \\
k
\end{array}\right)\Delta^{m+1} f(t-kh)
, \ \ \
\end{align*}
\begin{align}\label{e20}
= \dfrac{1}{\Gamma(m-\alpha+1)}\int\limits_{t-L}^{t}(t-\tau)^{m-\alpha}f^{(m+1)}(\tau)d\tau, \ \ nh=L 
\end{align}
Finally using (\ref{e17}) and (\ref{e20}) we conclude the limit (\ref{e6}).

\end{proof}

\begin{proposition}\label{p2} (Fractional derivative of a power function)\\
Let $\alpha> 0$, $L>0$,  $m$ an integer such that $m-1<\alpha<m$ and $f(t)=t^{n}$, then we have
\begin{equation}\label{e9}
\ _{\ \ L}^{^{MG}}D_{t}^{\alpha}(t^{n})=\sum\limits_{k=0}^{n}\dfrac{n!L^{k-\alpha}(t-L)^{n-k}}{(n-k)!\Gamma(k-\alpha+1)}.
\end{equation} 
\end{proposition}

\begin{proof}
If $n\leq m$, then $f^{(m+1)}\equiv 0$, substituting in (\ref{e6}) yields the relation (\ref{e9}).
If $n> m$, then from (\ref{e6}) and using $(n-m-1)$ successive integration by part we obtain the relation (\ref{e9}).

\end{proof}

\begin{corollary}\label{c1} (Fractional derivative of a constant function)\\
If $f$ is a constant function (i.e. $f(t)=C$ for all $t\in [a-L,b]$), then $\ _{\ \ L}^{^{MG}}D_{t}^{\alpha}f$ is a constant function for all $t\in [a,b]$).  Furthermore we have

\begin{equation}\label{e5}
\ _{\ \ L}^{^{MG}}D_{t}^{\alpha}f=\dfrac{C}{L^{\alpha}.\Gamma(1-\alpha)}.
\end{equation}

\end{corollary}

\begin{proof}
Suppose that $f$ is a constant function ( $f(t)=C$ for all $t\in [a-L,b]$) , then  
$f^{(k)}(t-L)=f^{(m+1)}(\tau)= 0$ for $k=1,2,...,m$ substituting in (\ref{e6}) yields the relation (\ref{e5})
\end{proof}

\begin{proposition}\label{p3} (Fractional derivative of the exponential function)\\
Let $\alpha> 0$, $L>0$ and  $m$ an integer such that $m-1<\alpha<m$, then we have
\begin{equation}\label{e10}
\ _{\ \ L}^{^{MG}}D_{t}^{\alpha}(e^{t})=\dfrac{E_{1,1-\alpha}(L)}{L^{\alpha}e^{L}}e^{t}.
\end{equation} 
\end{proposition}
\begin{proof}
We have $e^{t}=\sum\limits_{p=0}^{\infty}\dfrac{t^{p}}{p!}$, then
\begin{eqnarray*}
\ _{\ \ L}^{^{MG}}D_{t}^{\alpha}(e^{t})&=\sum\limits_{p=0}^{\infty}\dfrac{\ _{\ \ L}^{^{MG}}D_{t}^{\alpha}(t^{p})}{p!}\ \ \ \ \ \ \ \ \ \ \ \ \ \ \ \  \ \ \ \ \ \ \ \ \ \ \ \ \ \ \  \ \ \ \ \ \ \ \ \ \ \ \ \ \ \  \\
&=\sum\limits_{p=0}^{\infty}\sum\limits_{k=0}^{p}\dfrac{L^{k-\alpha}(t-L)^{p-k}}{(p-k)!\Gamma(k-\alpha+1)} \ \ \ \ \ \ \ \ \ \ \ \ \ \ \ \ \ \ \ \ \ \ \ \ \ \ \ \ \ \  \\
&=\dfrac{L^{-\alpha}}{\Gamma(1-\alpha)}e^{t-L}+\dfrac{L^{1-\alpha}}{\Gamma(2-\alpha)}e^{t-L}+\dfrac{L^{2-\alpha}}{\Gamma(3-\alpha)}e^{t-L}+...\\
&=e^{t-L}\sum\limits_{p=0}^{\infty}\dfrac{L^{p-\alpha}}{\Gamma(p+1-\alpha)}\ \ \ \ \ \ \ \ \ \ \ \ \ \ \ \  \ \ \ \ \ \ \ \ \ \ \ \ \  \ \ \ \ \ \ \ \ \  \\
&=e^{t-L}L^{-\alpha}\sum\limits_{p=0}^{\infty}\dfrac{L^{p}}{\Gamma(p+1-\alpha)} \ \ \ \ \ \ \ \ \ \ \ \  \ \ \ \ \ \ \ \ \ \ \ \ \ \ \  \ \ \ \ \ \ \ \\
&=\dfrac{E_{1,1-\alpha}(L)}{L^{\alpha}e^{L}}e^{t}.\ \ \ \ \ \ \ \ \ \ \ \  \ \ \ \ \ \ \ \ \ \ \ \ \ \ \  \ \ \ \ \ \ \  \ \ \ \ \ \ \ \ \ \ \ \  \ \  \ \ \ \ 
\end{eqnarray*}
\end{proof}

\subsection{An interpolation property}
The fractional-order derivative mean the extension of the integer-order derivative to real order, in other word the fractional-order operator $D_{t}^{\alpha}$ is an interpolation of the integer-order operator $\dfrac{d^{n}}{dt^{n}}$, so it is naturally to ask the question: is the proposed Gr$\ddot{\text{u}}$nwald-Letnikov fractional-order derivative with fixed memory length an extension of the integer-order derivative? The following proposition gives the answer to this question.

\begin{proposition}\label{p1}
Let $\alpha> 0$, $L>0$,  $m$ an integer such that $m-1<\alpha<m$ and $f$ an $m$-differentiable function on $[a-L,b]$. Then for all $t\in [a,b]$ we have
\begin{equation} \label{e3}
\lim\limits_{\alpha\rightarrow m }  \ _{\ \ L}^{^{MG}}D_{t}^{\alpha}f(t)=f^{(m)}(t).
\end{equation}
And
\begin{equation} \label{e4}
\lim\limits_{\alpha\rightarrow m-1 }  \ _{\ \ L}^{^{MG}}D_{t}^{\alpha}f(t)=f^{(m-1)}(t).
\end{equation}
\end{proposition}

\begin{proof}
We have
\begin{eqnarray*}
\lim\limits_{\alpha\rightarrow m }  \ _{\ \ L}^{^{MG}}D_{t}^{\alpha}f(t)&=\lim\limits_{\alpha\rightarrow m }\left[ \lim\limits_{h\rightarrow 0}\dfrac{1}{h^{\alpha}}\sum\limits_{k=0}^{\frac{L}{h}}(-1)^{k}\dfrac{\Gamma(\alpha+1)}{k!\Gamma(\alpha-k+1)}f(t-kh)\right],\\
&=\lim\limits_{h\rightarrow 0}\sum\limits_{k=0}^{\frac{L}{h}}(-1)^{k}\lim\limits_{\alpha\rightarrow m }\left[\dfrac{1}{h^{\alpha}}\dfrac{\Gamma(\alpha+1)}{k!\Gamma(\alpha-k+1)}\right] f(t-kh),
\end{eqnarray*}

when $k>m$ then $\lim\limits_{\alpha\rightarrow m }\dfrac{\Gamma(\alpha+1)}{k!\Gamma(\alpha-k+1)}=\left( \begin{array}{c}
m      \\
k
\end{array}\right)=0$, thus
\begin{eqnarray*}
\lim\limits_{\alpha\rightarrow m }  \ _{\ \ L}^{^{MG}}D_{t}^{\alpha}f(t)&=\lim\limits_{h\rightarrow 0}\dfrac{1}{h^{m}}\sum\limits_{k=0}^{m}(-1)^{k}\dfrac{\Gamma(m+1)}{k!\Gamma(m-k+1)}f(t-kh),\\
&=\lim\limits_{h\rightarrow 0}\dfrac{1}{h^{m}}\sum\limits_{k=0}^{m}(-1)^{k}\dfrac{m!}{k!(m-k)!}f(t-kh),\ \ \ \ \ \ \\
&=f^{(m)}(t).\ \ \ \ \ \ \ \ \ \ \ \ \ \ \ \ \ \ \ \ \ \ \ \ \ \ \ \ \ \ \ \ \ \ \ \ \ \ \ \ \ \ \ \ \ 
\end{eqnarray*}
By similar procedure we prove that 
\begin{equation*} 
\lim\limits_{\alpha\rightarrow m-1 }  \ _{\ \ L}^{^{MG}}D_{t}^{\alpha}f(t)=f^{(m-1)}(t).
\end{equation*}
Thus the fractional-order operator $ \ _{\ \ L}^{^{MG}}D_{t}^{\alpha}$ is an interpolation of the integer-order operator $\dfrac{d^{n}}{dt^{n}}$.
\end{proof}
\begin{example}
Let $\alpha> 0$, $L>0$,  $m$ an integer such that $m-1< \alpha<m$ and suppose that $f(t)=C$ is a constant function  then from corollary \textbf{(\ref{c1})} we have
\begin{equation}\label{e7}
\lim\limits_{\alpha\rightarrow m } \ _{\ \ L}^{^{MG}}D_{t}^{\alpha}f(t)=\lim\limits_{\alpha\rightarrow m } \dfrac{C}{L^{\alpha}.\Gamma(1-\alpha)}=0=f^{(m)}(t),
\end{equation} 
and
\begin{equation}\label{e8}
\lim\limits_{\alpha\rightarrow 0 } \ _{\ \ L}^{^{MG}}D_{t}^{\alpha}f(t)=\lim\limits_{\alpha\rightarrow 0 } \dfrac{C}{L^{\alpha}.\Gamma(1-\alpha)}=C=f(t).
\end{equation} 
\end{example}

\subsection{Derivative of periodic function}
The main result of this paper as previously announced is the introduction of new modified definition of fractional derivative with finite fixed memory length which preserves the periodicity, this result is summarized in the following theorem
\begin{theorem}
Let $\alpha> 0$, $L>0$,  $m$ an integer such that $m-1< \alpha<m$ and $f$  a function defined on $[a-L,b]$ such that $\ _{\ \ L}^{^{MG}}D_{t}^{\alpha}f$ exists.
If $f$ is a periodic function with period $T$ (i.e. $f(t+T)=f(t)$ for all $t\in [a-L,b]$), then $\ _{\ \ L}^{^{MG}}D_{t}^{\alpha}f$ is a periodic function with the same period $T$.
\end{theorem}
\begin{proof}
Suppose that $f$ is a periodic function with period $T$, then we have 

\begin{eqnarray*} \label{e2}
_{\ \ L}^{^{MG}}D_{t}^{\alpha}f(t+T)&=\lim\limits_{h\rightarrow 0}\dfrac{1}{h^{\alpha}}\sum\limits_{k=0}^{\frac{L}{h}}(-1)^{k}\dfrac{\Gamma(\alpha+1)}{k!\Gamma(\alpha-k+1)}f(t+T-kh),\\
&= \lim\limits_{h\rightarrow 0}\dfrac{1}{h^{\alpha}}\sum\limits_{k=0}^{\frac{L}{h}}(-1)^{k}\dfrac{\Gamma(\alpha+1)}{k!\Gamma(\alpha-k+1)}f(t-kh).\ \ \ \ \ \ \\
&= \ _{\ \ L}^{^{MG}}D_{t}^{\alpha}f(t). \ \ \ \ \ \ \ \ \ \ \ \ \ \ \ \ \ \ \ \ \ \ \ \ \ \ \ \ \ \ \  \ \ \ \ \ \ \ \ \ \ \ \ \ \
\end{eqnarray*} 
Hence $\ _{\ \ L}^{^{MG}}D_{t}^{\alpha}f$ is a periodic function with period $T$.
\end{proof}

\begin{example} (Fractional derivative of sinus function)\\
We have $sin(t)=\sum\limits_{p=0}^{\infty}(-1)^{p}\dfrac{t^{2p+1}}{(2p+1)!}$, then
\begin{align*}
\ _{\ \ L}^{^{MG}}D_{t}^{\alpha}sin(t)=\sum\limits_{p=0}^{\infty}(-1)^{p}\dfrac{\ _{\ \ L}^{^{MG}}D_{t}^{\alpha}(t^{2p+1})}{(2p+1)!}
\end{align*}
using the relation \textbf{(\ref{e9})} we obtain
\begin{eqnarray}
\ _{\ \ L}^{^{MG}}D_{t}^{\alpha}sin(t)&=\sum\limits_{p=0}^{\infty}\sum\limits_{k=0}^{2p+1}(-1)^{p}\dfrac{L^{k-\alpha}(t-L)^{2p+1-k}}{(2p+1-k)!\Gamma(k-\alpha+1)} \ \ \ \ \ \ \ \ \ \ \ \ \ \ \ \ \ \ \ \ \ \ \ \ \ \ \ \ \ \ \ \ \ \ \ \ \ \ \ \ \nonumber \\
&=\dfrac{L^{-\alpha}}{\Gamma(1-\alpha)}sin(t-L)+\dfrac{L^{1-\alpha}}{\Gamma(2-\alpha)}cos(t-L)-\dfrac{L^{2-\alpha}}{\Gamma(3-\alpha)}sin(t-L)\ \ \ \ \ \ \ \ \ \nonumber \\
&-\dfrac{L^{3-\alpha}}{\Gamma(4-\alpha)}cos(t-L)+...\ \ \ \ \ \ \ \ \ \ \ \ \ \ \ \ \ \ \ \ \ \ \ \ \ \ \ \ \ \ \ \ \ \ \ \ \ \ \ \ \ \ \ \ \ \ \ \ \ \ \ \ \ \ \nonumber  \\
&=sin(t-L)\sum\limits_{p=0}^{\infty}(-1)^{p}\dfrac{L^{2p-\alpha}}{\Gamma(2p+1-\alpha)}+cos(t-L)\sum\limits_{p=0}^{\infty}(-1)^{p}\dfrac{L^{2p+1-\alpha}}{\Gamma(2p+2-\alpha)}\nonumber \\
&=L^{-\alpha}sin(t-L)E_{2,1-\alpha}(-L^{2})+L^{1-\alpha}cos(t-L)E_{2,2-\alpha}(-L^{2})\ \ \ \ \ \ \  \ \ \ \ \ \ \nonumber \\
&=a \sin(t-L)+b \cos(t-L),\ \ \ \ \ \ \  \ \ \ \ \ \ \ \ \ \ \ \ \  \ \ \ \ \ \ \ \ \ \ \ \ \  \ \ \ \ \ \ \ \ \ \ \ \ \  \ \ \ \ \label{e26}
\end{eqnarray}
where $a=L^{-\alpha}E_{2,1-\alpha}(-L^{2})$ and $b=L^{1-\alpha}E_{2,2-\alpha}(-L^{2})$.\\
Clearly $\ _{\ \ L}^{^{MG}}D_{t}^{\alpha}sin(t)$ is a periodic function with period $2\pi$.\\
Figure \ref{fg2} shows the modified fractional-order derivative $\ _{\ \ L}^{^{MG}}D_{t}^{\alpha}sin(t)$ of sine function for $L=30$ and some value of $\alpha$ which confirms the analytic result obtained.
\end{example}
\begin{example} (Fractional derivative of cosines function)\\
We have $cos(t)=\sum\limits_{p=0}^{\infty}(-1)^{p}\dfrac{t^{2p}}{(2p)!}$, then
\begin{eqnarray}
\ _{\ \ L}^{^{MG}}D_{t}^{\alpha}cos(t)&=\sum\limits_{p=0}^{\infty}(-1)^{p}\dfrac{\ _{\ \ L}^{^{MG}}D_{t}^{\alpha}(t^{2p})}{(2p)!}\ \ \ \ \ \ \ \ \ \ \ \ \ \ \ \  \ \ \ \ \ \ \ \ \ \ \ \ \ \ \  \ \ \ \ \ \ \ \ \ \ \ \ \ \ \ \ \ \ \ \ \ \ \ \ \ \ \ \ \  \nonumber  \\
&=\sum\limits_{p=0}^{\infty}\sum\limits_{k=0}^{2p}(-1)^{p}\dfrac{L^{k-\alpha}(t-L)^{2p-k}}{(2p-k)!\Gamma(k-\alpha+1)} \ \ \ \ \ \ \ \ \ \ \ \ \ \ \ \ \ \ \ \ \ \ \ \ \ \ \ \ \ \ \ \ \ \ \ \ \ \ \ \ \nonumber  \\
&=\dfrac{L^{-\alpha}}{\Gamma(1-\alpha)}cos(t-L)-\dfrac{L^{1-\alpha}}{\Gamma(2-\alpha)}sin(t-L)-\dfrac{L^{2-\alpha}}{\Gamma(3-\alpha)}cos(t-L)\ \ \ \ \ \ \ \ \  \nonumber  \\
&+\dfrac{L^{3-\alpha}}{\Gamma(4-\alpha)}sin(t-L)+...\ \ \ \ \ \ \ \ \ \ \ \ \ \ \ \ \ \ \ \ \ \ \ \ \ \ \ \ \ \ \ \ \ \ \ \ \ \ \ \ \ \ \ \ \ \ \ \ \ \ \ \ \ \  \nonumber  \\
&=cos(t-L)\sum\limits_{p=0}^{\infty}(-1)^{p}\dfrac{L^{2p-\alpha}}{\Gamma(2p-\alpha)}-sin(t-L)\sum\limits_{p=0}^{\infty}(-1)^{p}\dfrac{L^{2p+1-\alpha}}{\Gamma(2p+2-\alpha)} \nonumber \\
&=L^{-\alpha}cos(t-L)E_{2,1-\alpha}(-L^{2})+L^{1-\alpha}sin(t-L)E_{2,2-\alpha}(-L^{2})\ \ \ \ \ \ \  \ \ \ \ \ \ \nonumber  \\
&=a \cos(t-L)-b \sin(t-L).\ \ \ \ \ \ \  \ \ \ \ \ \ \ \ \ \ \ \ \  \ \ \ \ \ \ \ \ \ \ \ \ \  \ \ \ \ \ \ \ \ \ \ \ \ \  \ \ \ \ \label{e27}
\end{eqnarray}
where $a=L^{-\alpha}E_{2,1-\alpha}(-L^{2})$ and $b=L^{1-\alpha}E_{2,2-\alpha}(-L^{2})$.\\
Clearly $\ _{\ \ L}^{^{MG}}D_{t}^{\alpha}cos(t)$ is a periodic function with period $2\pi$.\\
\end{example}
\begin{remark}
When the memory length $L$ increases its influence on the derivative decreases, as shown in Figure \ref{fg3}, thus in order to obtain an independent fractional derivative from $L$ it is recommended to chose $L$ greater as possible. 
\end{remark}
\subsection{Fractional-order autonomous system with exact periodic solution}
As previously mentioned the fractional-order autonomous system expressed in term of classical fractional derivatives can not have exact periodic solutions.
In this subsection we prove that the fractional-order autonomous system expressed in term of our proposed fractional derivative can have exact periodic solution. For this purpose we give the following example.
\begin{example}
Let consider the following linear fractional-order autonomous system 

\begin{equation}\label{e24}
\ _{\ \ 2\pi}^{^{MG}}D_{t}^{\alpha}X(t)=AX(t)
\end{equation}

where $X(t)\in R^{2}$ and 
$
A=\left[ \begin{array}{c c}
a& -b \\
b& a
\end{array}\right]
$\\ 
with $a=(2\pi)^{-\alpha}E_{2,1-\alpha}(-4\pi^{2})$ and $b=(2\pi)^{1-\alpha}E_{2,2-\alpha}(-4\pi^{2})$.\\
The vector function 
$$X(t)=c\left( \begin{array}{c}
cos(t)\\
sin(t)
\end{array}\right),\  c\in R$$ 
is an exact $2\pi$-periodic solution of (\ref{e24}), namely we have  
$$\ _{\ \ 2\pi}^{^{MG}}D_{t}^{\alpha}X(t)=c\left( \begin{array}{c}
\ _{\ \ 2\pi}^{^{MG}}D_{t}^{\alpha}cos(t)\\
\\
\ _{\ \ 2\pi}^{^{MG}}D_{t}^{\alpha}sin(t)
\end{array}\right)$$
then from (\ref{e26}) and (\ref{e27}) we get

\begin{align}
\ _{\ \ 2\pi}^{^{MG}}D_{t}^{\alpha}X(t)&=c\left( \begin{array}{c}
a \cos(t-2\pi)-b \sin(t-2\pi)\\
\\
a \sin(t-2\pi)+b \cos(t-2\pi)
\end{array}\right) \nonumber\\
&=c\left( \begin{array}{c}
a \cos(t)-b \sin(t)\nonumber\\
\\
a \sin(t)+b \cos(t)
\end{array}\right)\nonumber\\
&=A X(t)\nonumber
\end{align}
then $X(t)=c\left( \begin{array}{c}
cos(t)\\
sin(t)
\end{array}\right)$ 
is an exact $2\pi$-periodic solution of (\ref{e24}).
\end{example}

\newpage
\begin{figure}
[ptb]
\begin{center}
\includegraphics[
height=1.7in,
width=5in
]%
{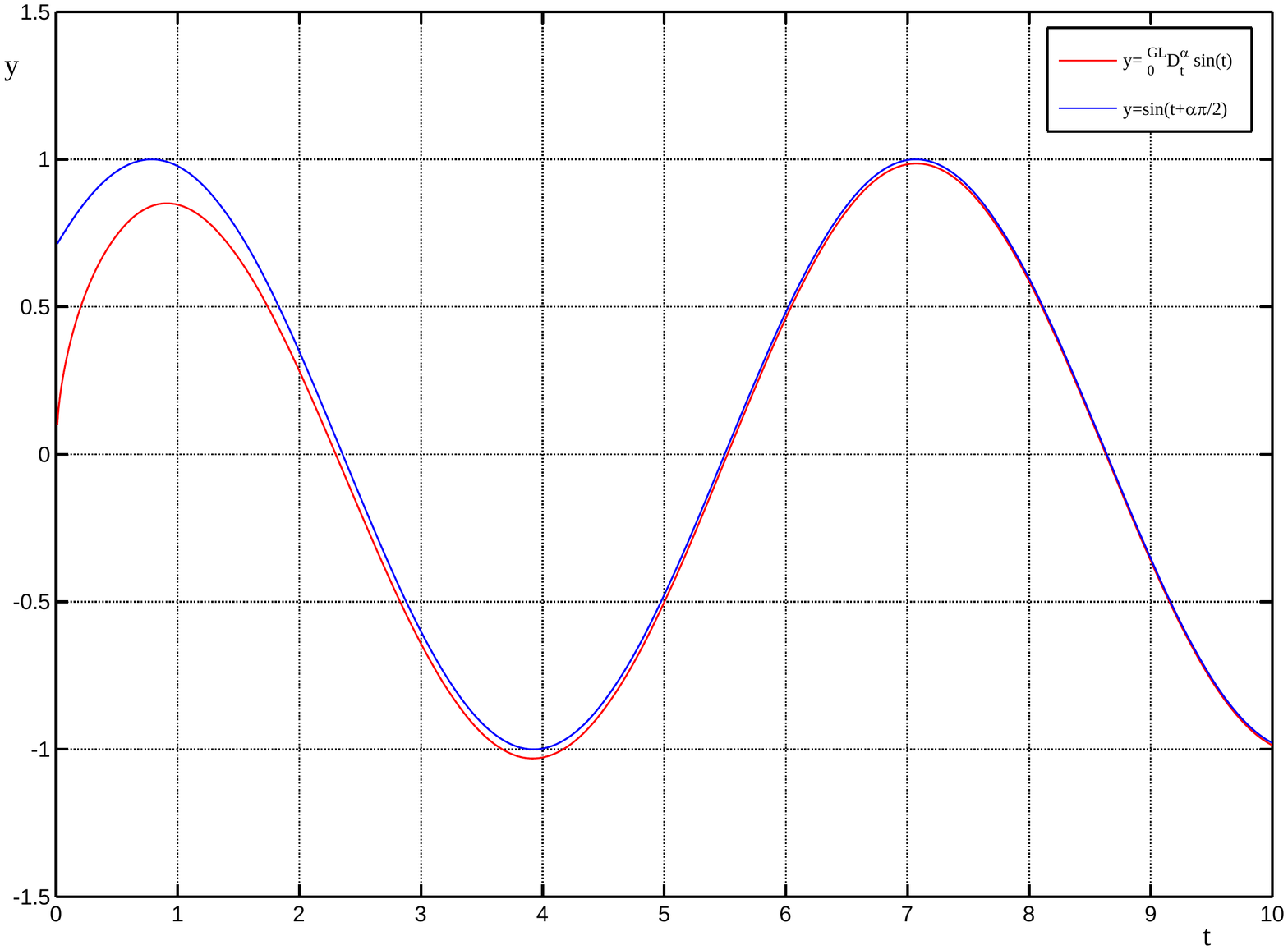}%
\caption{Fractional derivative of sine function $\ _{\ \ }^{^{GL}}D_{t}^{\alpha}sin(t)$ for $\alpha=0.5$}%
\label{fg1}%
\end{center}
\end{figure}

\begin{figure}
[ptb]
\begin{center}
\includegraphics[
height=1.7in,
width=5in
]%
{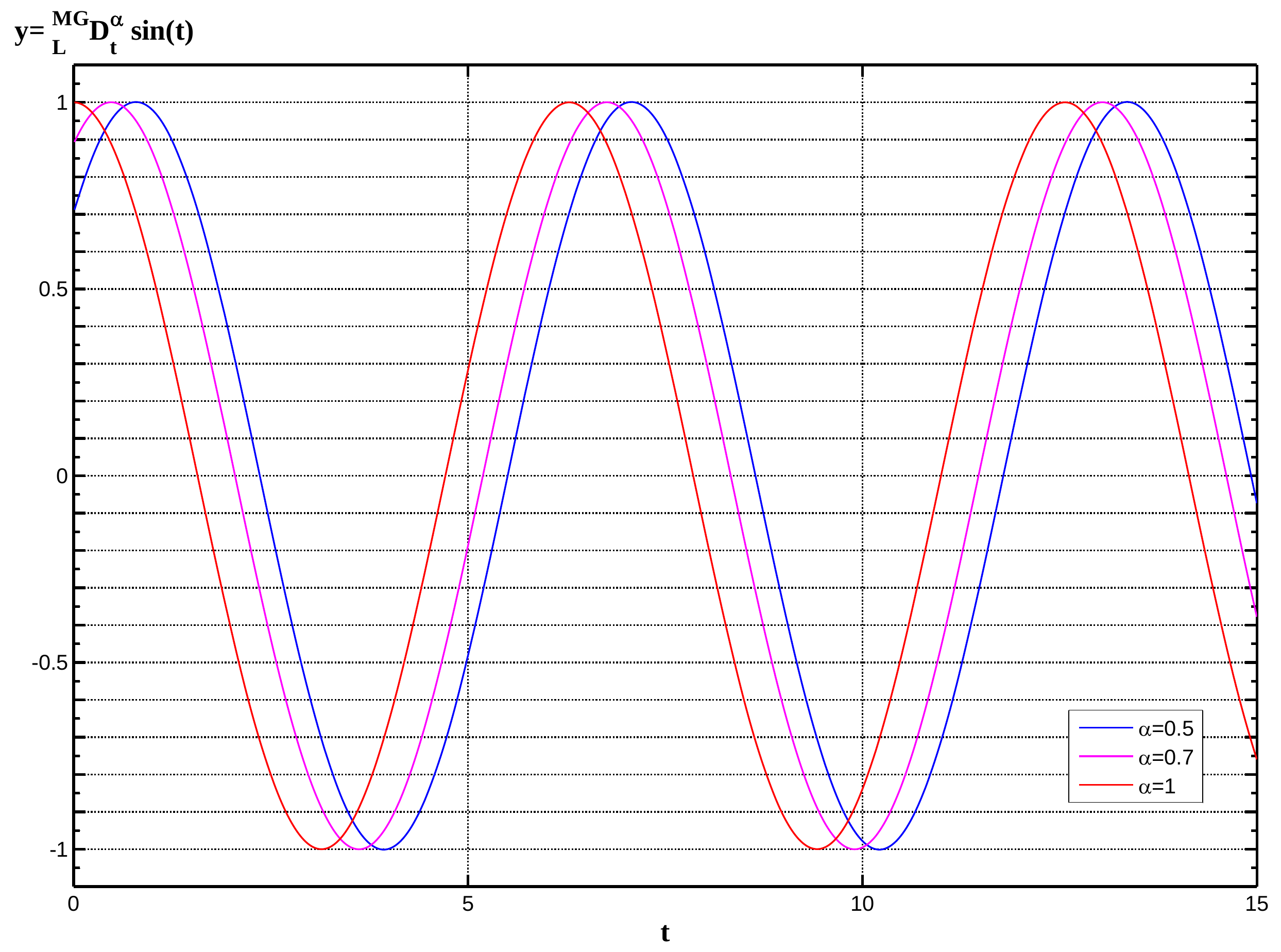}%
\caption{Fractional derivative of sine function $\ _{\ \ L}^{^{MG}}D_{t}^{\alpha}sin(t)$ for $L=30$ and some value of $\alpha$.}%
\label{fg2}%
\end{center}
\end{figure}

\begin{figure}
[ptb]
\begin{center}
\includegraphics[
height=2in,
width=5in
]%
{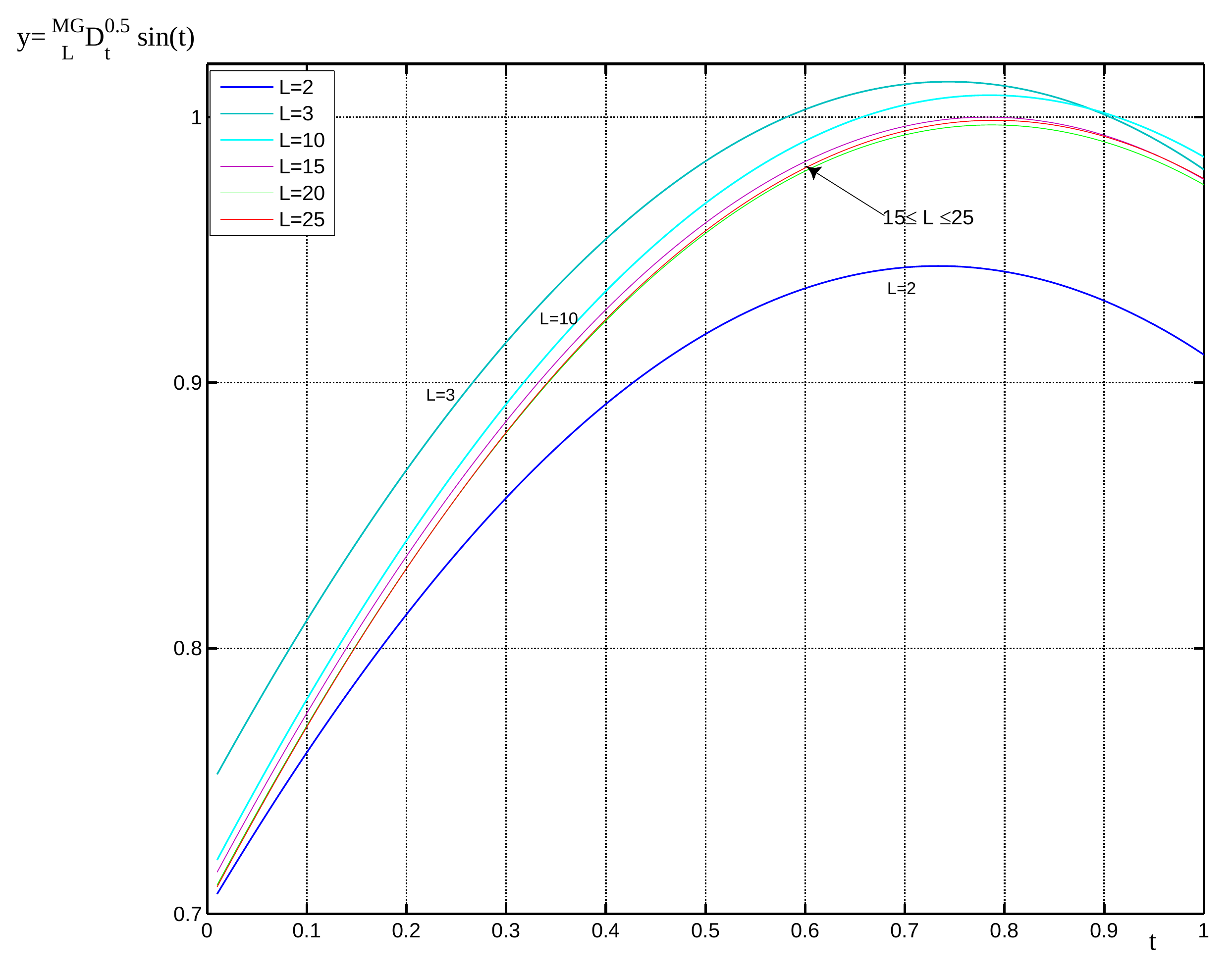}%
\caption{Influence of the memory length $L$ on the fractional derivative $\ _{\ \ L}^{^{MG}}D_{t}^{0.5}sin(t)$.}%
\label{fg3}%
\end{center}
\end{figure}
\newpage
\section{Conclusion and discussion}
Although the idea of fixed memory length is inspired from the short-memory principle introduced by I.Podlubny (1999) for numerical needs, there is a significantly difference between this two ideas near the starting point. Namely for the calculation of the fractional derivative of a certain function $f$, on the interval $[a,b]$, using the first idea, the  memory  length $L$ is fixed for all $t \in [a,b ]$, as previously mentioned, but using  the second idea, if $t \in [a,a+L]$ then the lower terminal is fixed at $a$, and the memory length $L$ is a function of $t$ ( $L=t-a$ ) and this do not allow the preservation of periodicity.
The modified definition of Gr$\ddot{\text{u}}$nwald-Letnikov fractional-order derivative reported in this paper posses two useful properties, the first is the preservation of periodicity as it is demonstrated and the second one is the short memory, which reduces considerably the cost of numerical computations. We have proven that contrary to fractional autonomous systems in term of classical fractional derivative the fractional autonomous systems in term of the modified fractional derivative can generate exact periodic solutions.
Now, this related question may be raised. If a differential equation is described based on the fractional operator proposed in this paper. What hypothesis on the equation are required for guaranteeing the existence and uniqueness of the solution? and what type of initial conditions are required? Finding the answer of this question and investigate other property of the proposed fractional derivative, can be an interesting topic for future research works.

\end{document}